\documentclass[12pt,a4]{amsart}
\usepackage[a4paper, left=28mm, right=28mm, top=28mm, bottom=34mm]{geometry}
\usepackage[all]{xy}
\usepackage{amsmath}
\usepackage{amssymb}
\usepackage{amsthm}
\usepackage{mathrsfs}
\usepackage{mathtools}
\usepackage{mathabx}
\theoremstyle{definition}
\newtheorem{thm}{Theorem}[section]
\newtheorem{lem}[thm]{Lemma}
\newtheorem{cor}[thm]{Corollary}
\newtheorem{prop}[thm]{Proposition}
\newtheorem{conj}[thm]{Conjecture}
\theoremstyle{definition}

\newtheorem{rem}[thm]{Remark}

\numberwithin{equation}{section}
\def\A{{\mathbb A}}

\def\Q{{\mathbb Q}}
\def\R{{\mathbb R}}
\def\Z{{\mathbb Z}}
\def\C{{\mathbb C}}

\def\Lie{\mathop{\mathrm{Lie}}\nolimits}

\def\Ker{\mathop{\mathrm{Ker}}\nolimits}

\def\Pic{\mathop{\mathrm{Pic}}\nolimits}

\def\GL{\mathop{\mathrm{GL}}\nolimits}
\def\SL{\mathop{\mathrm{SL}}\nolimits}

\def\Tr{\mathop{\mathrm{Tr}}\nolimits}
\def\det{\mathop{\mathrm{det}}\nolimits}
\def\dim{\mathop{\mathrm{dim}}\nolimits}

\def\L{\mathscr{L}}
\def\H{\mathscr{H}}
\def\g{\mathfrak{g}}

\def\N{\mathscr{N}}
\def\Res{\mathop{\mathrm{Res}}\nolimits}
\def\GSpin{\mathop{\mathrm{GSpin}}\nolimits}
\def\CH{\mathop{\mathrm{CH}}\nolimits}

\def\exp{\mathop{\mathrm{exp}}\nolimits}
\def\Sp{\mathop{\mathrm{Sp}}\nolimits}
\def\SO{\mathop{\mathrm{SO}}\nolimits}
\def\O{\mathop{\mathrm{O}}\nolimits}
\def\Mp{\mathop{\mathrm{Mp}}\nolimits}
\def\Sym{\mathop{\mathrm{Sym}}\nolimits}

\def\O{\mathop{\mathrm{O}}\nolimits}
\def\S{\textbf{\textsl{S}}}

\newcommand{\transp}[1]{{}^{t}\!{#1}}
\newcommand{\defeq}{\vcentcolon=}

\begin{document}

\title[Modularity of special cycles on orthogonal Shimura varieties]
{The modularity of special cycles on orthogonal Shimura varieties over totally real fields under the Beilinson-Bloch conjecture}

\author{Yota Maeda}
\address{Department of Mathematics, Faculty of Science, Kyoto University, Kyoto 606-8502, Japan}
\email{y.maeda@math.kyoto-u.ac.jp}

\date{October 3, 2019}



\date{\today}

\maketitle

\begin{abstract}
We study special cycles on a Shimura variety of orthogonal type over a totally real field of degree $d$ associated with a quadratic form in $n+2$ variables whose signature is $(n,2)$ at $e$ real places and $(n+2,0)$ at the remaining $d-e$ real places for $1\leq e <d$.
Recently, these cycles were constructed by Kudla and Rosu-Yott and they proved that the generating series of special cycles in the cohomology group is a Hilbert-Siegel modular form of half integral weight.
We prove that, assuming the Beilinson-Bloch conjecture on the injectivity of the higher Abel-Jacobi map, the generating series of special cycles of codimension $er$ in the Chow group is a Hilbert-Siegel modular form of genus $r$ and weight $1+n/2$.
Our result is a generalization of \textit{Kudla's modularity conjecture}, solved by Yuan-Zhang-Zhang unconditionally when $e=1$.
\end{abstract}

\section{Introduction}

In this paper, we prove that, assuming the Beilinson-Bloch conjecture on the injectivity of the higher Abel-Jacobi map, the generating series of special cycles in the Chow groups of a Shimura variety of orthogonal type is a Hilbert-Siegel modular form of half integral weight.
These cycles were constructed by Kudla \cite{Remarks} and Rosu-Yott \cite{RosuYott}.

Historically, Kudla and Millson studied the cohomology groups in \cite{KM}.
Kudla conjectured the modularity of the generating series of special cycles in the Chow groups in \cite{SpecialKudla} and he proved it for one-codimensional Chow cycles, using the results of Borcherds \cite{Borcherds}.
This conjecture is often called \textit{Kudla's modularity conjecture}.
In his thesis \cite{Thesis}, Wei Zhang proved it for higher codimensional Chow cycles on Shimura varieties of orthogonal type associated with a quadratic form of signature $(n,2)$ over $\Q$ by his modularity criterion.
His criterion works only over $\Q$ because its proof depends on the results of Borcherds \cite{BorcherdsLift}.
Yuan-Zhang-Zhang \cite{YZZ} extended Wei Zhang's results \cite{Thesis} to totally real fields.
Their proof is similar to Wei Zhang's proof over $\Q$ in view of using induction on the codimension of Chow cycles and calculating element-wise modularity.

Recently, Kudla \cite{Remarks} and Rosu-Yott \cite{RosuYott} generalized Kudla-Millson's work by changing the signature of the quadratic form.
Rosu-Yott \cite{RosuYott} studied special cycles in the cohomology groups only, so did not generalize Yuan-Zhang-Zhang's work.
In this paper, we shall generalize the results of Yuan-Zhang-Zhang \cite{YZZ} under the Beilinson-Bloch conjecture.
In the same setting as \cite{Remarks}, \cite{RosuYott} and assuming the Beilinson-Bloch conjecture on the injectivity of the higher Abel-Jacobi map, we prove the modularity of the generating series of special cycles in the Chow groups.
(For the precise statement, see Theorem \ref{MainTheorem1} and Theorem \ref{MainTheorem2}.)

After the first version of this paper was written, the author learned that Kudla independly obtained similar results in his recent preprint \cite{Remarks}.
His results and proof are different from ours.
More precisely, in \cite{Remarks}, he assumed the Beilinson-Bloch conjecture for Chow cycles of codimension $er$, and proved the absolute convergence and the modularity of the generating series.
In contrast, even if $r\geq 2$, we assume the Beilinson-Bloch conjecture for Chow cycles of codimension $e$ only.
However, we cannot prove the absolute convergence.
Assuming the absolute convergence, we prove the modularity by induction on $r$ by the methods of \cite{YZZ}.
(For details, see Remark \ref{Kudla}.)

\subsection{Special cycles on Shimura varieties of orthogonal type}
\label{Subsection:Special cycles}

Before giving the statement of our results,
we briefly recall the setting of Kudla \cite{OrthogonalKudla}, \cite{Remarks} and Rosu-Yott \cite{RosuYott}.

Let $d$ and $e$ be positive integers satisfying $1\leq e<d$.
Let $F$ be a totally real field of degree $d$ with real embeddings $\sigma_1,\dots,\sigma_d$.
Let $V$ be a non-degenerate quadratic space of dimension $n+2$ over $F$ whose signature is $(n,2)$ at $\sigma_1,\dots,\sigma_e$ and $(n+2,0)$ at $\sigma_{e+1},\dots,\sigma_d$.
We put $V_{\sigma_i,\C}\defeq V\otimes_{F,\sigma_i}\C$ and $\mathbb{P}(V_{\sigma_i,\C})\defeq(V_{\sigma_i,\C}\backslash \{0\})/\C^{\times}$.
Let $D_i \subset \mathbb{P}(V_{\sigma_i,\C})$ be the Hermitian symmetric domain defined as follows:
\[D_i\defeq\{v \in V_{\sigma_i,\C}\backslash \{0\} \mid \langle v,v \rangle=0, \ \langle v,\bar{v} \rangle<0\}/\C^\times \quad (1\leq i\leq e).\]
We put $D\defeq D_1\times \dots \times D_e$.
Let $\GSpin(V)$ be the general spin group of $V$ over $F$, which is a connected reductive group over $F$.
We put $G\defeq\Res_{F/\Q}\GSpin(V)$ and consider the Shimura varieties associated with $(G,D)$.
Then, for any open compact subgroup $K_f\subset G(\A_f)$, the Shimura datum ($G$,$D$) gives a Shimura variety $M_{K_f}$ over $\C$, whose $\C$-valued points are given as follows:
\[M_{K_f}(\C)=G(\Q) \backslash (D \times G(\A_f))/K_f.\]
Here $\A_f$ is the ring of finite ad\`{e}les of $\Q$.
We remark that $M_{K_f}$ has a canonical model  over a number field called the reflex field.
Hence $M_{K_f}$ is canonically defined over $\overline{\Q}$.
In this paper, $\overline{\Q}$ is an algebraic closure of $\Q$ embedded in $\C$.
By abuse of notation, in this paper, the canonical model of $M_{K_f}$ over $\overline{\Q}$ is also denoted by the same symbol $M_{K_f}$.
Then the Shimura variety $M_{K_f}$ is a projective variety over $\overline{\Q}$ since $1\leq e<d$.
It is a smooth variety over $\overline{\Q}$ if $K_f$ is sufficiently small.

For $i=1,\dots,e$, let $\L_i\in \Pic(D_i)$ be the line bundle which is the restriction of $\mathscr{O}_{\mathbb{P}(V_{\sigma_i,\C})}(-1)$ to $D_i$.
By pulling back to $D$, we get $p^*_i\L_i\in \Pic(D)$,
where $p_i\colon D \to D_i$ are the projection maps.
These line bundles descend to
$\L_{K_f,i} \in \Pic(M_{K_f})\otimes_{\Z}\Q$ and thus we obtain
$\L\defeq\L_{K_f,1}\otimes\dots\otimes\L_{K_f,e}$ on $M_{K_f}$.

We shall define special cycles following Kudla \cite{OrthogonalKudla}, \cite{Remarks} and Rosu-Yott \cite{RosuYott}.
Let $W\subset V$ be a totally positive subspace over $F$.
We denote $G_W\defeq\Res_{F/\Q}\GSpin(W^\perp)$.
Let $D_W\defeq D_{W,1}\times\dots\times D_{W,e}$ be the Hermitian
symmetric domain associated with $G_W$, where
\[D_{W,i}\defeq\{w \in D_i \mid \forall v \in W_{\sigma_i}, \ \langle v,w \rangle=0\} \quad (1\leq i\leq e).\]
Then we have an embedding of Shimura data $(G_W,D_W)\hookrightarrow (G,D)$.
For any open compact subgroup $K_f \subset G(\A_f)$ and $g \in G(\A_f)$,
we have an associated Shimura variety $M_{gK_fg^{-1},W}$ over $\C$:
\[M_{gK_fg^{-1},W}(\C)=G_W(\Q) \backslash (D_W \times G_W(\A_f))/(gK_fg^{-1} \cap G_W(\A_f)).\]
Assume that $K_f$ is neat so that the following morphism
\begin{align*}
M_{gK_fg^{-1},W}(\C) &\to M_{K_f}(\C)\\
[\tau,h]&\mapsto [\tau,hg]
\end{align*}
is a closed embedding \cite[Lemma 4.3]{Remarks}.
Let $Z(W,g)_{K_f}$ be the image of this morphism.
We consider $Z(W,g)_{K_f}$ as an algebraic cycle of codimension $e\dim_FW$ on $M_{K_f}$ defined over $\overline{\Q}$.

For any positive integer $r$ and $x=(x_1,\dots,x_r) \in V^r$, let $U(x)$ be the $F$-subspace of $V$ spanned by $x_1,\dots,x_r$.
We define the \textit{special cycle} in the Chow group
\[Z(x,g)_{K_f} \in \CH^{er}(M_{K_f})_{\C}\defeq\CH^{er}(M_{K_f})\otimes_{\Z}\C\]
by
\[Z(x,g)_{K_f}\defeq Z(U(x),g)_{K_f}(\mathrm{c}_1(\L^{\vee}_{K_f,1})\cdots\mathrm{c}_1(\L^{\vee}_{K_f,e}))^{r-\dim U(x)}\]
if $U(x)$ is totally positive.
Otherwise, we put $Z(x,g)_{K_f}\defeq 0$.

For a Bruhat-Schwartz function $\phi_f\in \S(V(\A_f)^r)^{K_f}$, \textit{Kudla's generating function} is defined to be the following formal power series with coefficients in $\CH^{er}(M_{K_f})_{\C}$ in the variable $\tau=(\tau_1,\dots,\tau_d)\in (\H_r)^d$:
\[Z_{\phi_f}(\tau)\defeq\sum_{x\in G(\Q)\backslash V^r}\sum_{g\in G_x(\A_f)\backslash G(\A_f)/K_f}\phi_f(g^{-1}x)Z(x,g)_{K_f}q^{T(x)}.\]
Here $G_x\subset G$ is the stabilizer of $x$, $\H_r$ is the Siegel upper half plane of genus $r$, $T(x)$ is the moment matrix $\frac{1}{2}((x_i,x_j))_{i,j}$, and
\[q^{T(x)}\defeq\exp(2\pi \sqrt{-1} \sum_{i=1}^d \Tr{\tau_iT(x)^{\sigma_i}}).\]

For a $\C$-linear map $\ell\colon\CH^{er}(M_{K_f})_{\C}\to \C$, we put
\[\ell(Z_{\phi_f})(\tau)\defeq\sum_{x\in G(\Q)\backslash V^r}\sum_{g\in G_x(\A_f)\backslash G(\A_f)/K_f}\phi_f(g^{-1}x)\ell(Z(x,g)_{K_f})q^{T(x)},\]
which is a formal power series with complex coefficients in the variable $\tau\in(\H_r)^d$
\subsection{The Beilinson-Bloch conjecture}
\label{Subsection:BB conjecture}

In 1980's, Beilinson and Bloch formulated a series of influential conjectures on algebraic cycles.
We review the statement of a part of the Beilinson-Bloch conjecture which is needed in the main theorem of this paper.
Our main reference is \cite{Beilinson}.
More generally, the Beilinson-Bloch conjecture is formulated in the theory of mixed motives, but we do not need the full version and need only a part of it for smooth projective varieties over number fields.
We recommend \cite{Jannsen} to the readers who want to know the Beilinson-Bloch conjecture in the theory of mixed motives.

In this subsection, let $k$ be a field of characteristic 0 embedded in $\C$.
Let $X$ be a smooth projective variety over $k$.
Let
\[cl^m\colon\CH^m(X)\to H^{2m}(X,\Q)\defeq H^{2m}(X(\C),\Q)\]
be the cycle map.
We put $\CH^m_{\mathrm{hom}}(X)\defeq\Ker(cl^m)$.

The following is a generalization of the Birch and Swinnerton-Dyer conjecture.
\begin{conj}(Beilinson-Bloch conjecture \cite[Conjecture 5.0]{Beilinson})
\label{Beilinson-Bloch1}
Assume that $k$ is a number field.
Then the group $\CH^m_{\mathrm{hom}}(X)$ is finitely generated and the rank of $\CH^m_{\mathrm{hom}}(X)$ is equal to the order of zero of the Hasse-Weil $L$-function  $L(H^{2m-1}_{\acute{e}t}(X\otimes_{k}\overline{k},\Q_{\ell}),s)$ at $s=m$ for any prime $\ell$.
\end{conj}
We recall another conjecture which is also considered as a part of the Beilinson-Bloch conjecture.
By Hodge theory, we have the Hodge decomposition
\[H^r(X,\C)=\bigoplus_{p+q=m}H^{p,q},\]
where $H^{p,q}\defeq H^q(X,\Omega^p)$ and a Hodge filtration $\bigl\{F^iH^m\bigl\}_{i=0}^m$ on $H^m$ is defined by
\[F^iH^m\defeq\bigoplus_{p\geq i}H^{p,m-p}.\]
The \textit{the $m$-th intermediate Jacobian of X} (or \textit{the Griffiths Jacobian of X}) is defined by
\[J^{2m-1}(X)\defeq H^{2m-1}(X,\C)/(F^mH^{2m-1}(X,\C)\oplus H^{2m-1}(X,\Z(m))).\]
Then we have \textit{the $m$-th higher Abel-Jacobi map}:
\[AJ^m\colon\CH^m_{\mathrm{hom}}(X)_{\Q}\defeq\CH^m_{\mathrm{hom}}(X)\otimes_{\Z}\Q \to J^{2m-1}(X)_{\Q}\defeq J^{2m-1}(X)\otimes_{\Z}{\Q}.\]

Here we can state another conjecture which is a part of version of the Beilinson-Bloch conjecture.
\begin{conj}(Beilinson-Bloch conjecture \cite[Lemma 5.6]{Beilinson})
\label{Beilinson-Bloch2}
The $m$-th higher Abel-Jacobi map $AJ^m$ is injective.
\end{conj}
Conjecture \ref{Beilinson-Bloch1} or Conjecture \ref{Beilinson-Bloch2} suggests the following is true.
Recall that $\overline{\Q}$ is an algebraic closure of $\Q$ embedded in $\C$.
\begin{conj}
\label{Beilinson-Bloch3}
Let $X$ be a smooth projective variety over $\overline{\Q}$.
If $H^{2m-1}(X,\Q)=0$, then $\CH^m_{\mathrm{hom}}(X)_{\Q}=0$.
In particular, the cycle map tensored with $\Q$
\[cl^m_{\Q}\colon\CH^m(X)_{\Q}\defeq\CH^m(X)\otimes_{\Z}\Q\to H^{2m}(X,\Q)\]
is injective.
\end{conj}
\begin{rem}
\label{Remark:r=1}
When $m=1$ and $X$ is a smooth projective curve over $\C$, the map $AJ^1$ is the usual Abel-Jacobi map, so we get an isomorphism between the Picard group and the Jacobian.
See \cite[Section 1.4]{Jannsen}.
From this, it is easy to see that Conjecture \ref{Beilinson-Bloch2} and Conjecture \ref{Beilinson-Bloch3} are true when $m=1$.
\end{rem}


\subsection{Main results}
Let notation be as in Section \ref{Subsection:Special cycles}.

If $n\geq 3$, our main result in this paper is:
\begin{thm}
\label{MainTheorem1}
Assume $n\geq 3$ and Conjecture \ref{Beilinson-Bloch3} for the Shimura variety $M_{K_f}$ for $m=e$.
Let $r\geq 1$ be a positive integer.
\begin{enumerate}
\item If $\ell\colon\CH^{er}(M_{K_f})_{\C}\to \C$ is a linear map over $\C$ such that $\ell(Z_{\phi_f})(\tau)$ is absolutely convergent, then $\ell(Z_{\phi_f})(\tau)$ defines a Hilbert-Siegel modular form of genus $r$ and weight $1+n/2$.

\item If $r=1$, for any linear map $\ell\colon\CH^{e}(M_{K_f})_{\C}\to \C$, the formal power series $\ell(Z_{\phi_f})(\tau)$ is absolutely convergent and we get a Hilbert modular form of weight $1+n/2$.
\end{enumerate}
\end{thm}

If $n\leq 2$, we need to embed $M_{K_f}$ into a larger Shimura variety.
Let $W$ be a totally positive quadratic space of dimension $\geq 3$ over $F$ and we put $G'\defeq\Res_{F/\Q}\GSpin(V\oplus W)$.
We may assume there is an open compact subgroup $K'_f\subset G'(\A_f)$ such that $K_f=K'_f\cap G(\A_f)$.
Let $M'_{K'_f}$ be the Shimura variety associated with $G'$ and $K'_f$ defined over $\overline{\Q}$.
Then we have an embedding of Shimura varieties $M_{K_f}\hookrightarrow M'_{K'_f}$ defined over $\overline{\Q}$.

\begin{thm}
\label{MainTheorem2}
Assume $n\leq 2$ and Conjecture \ref{Beilinson-Bloch3} for the larger Shimura variety $M'_{K'_f}$ for $m=e$.
Let $r\geq 1$ be a positive integer.
\begin{enumerate}
\item If $\ell\colon\CH^{er}(M_{K_f})_{\C}\to \C$ is a linear map over $\C$ such that $\ell(Z_{\phi_f})(\tau)$ is absolutely convergent, then $\ell(Z_{\phi_f})(\tau)$ defines a Hilbert-Siegel modular form of genus $r$ and weight $1+n/2$.

\item If $r=1$, for any linear map $\ell\colon\CH^{e}(M_{K_f})_{\C}\to \C$, the formal power series $\ell(Z_{\phi_f})(\tau)$ is absolutely convergent and we get a Hilbert modular form of weight $1+n/2$.
\end{enumerate}
\end{thm}

\begin{rem}
If $\ell$ factors through a linear map $\ell'\colon H^{2er}(M_{K_f},\C)\to\C$, Theorem \ref{MainTheorem1} and Theorem \ref{MainTheorem2} were proved unconditionally by Kudla \cite[Section 5.3]{Remarks} and Rosu-Yott \cite[Theorem 1.1]{RosuYott}.
\end{rem}

\begin{rem}
When $e=1$, we recover the results of Yuan-Zhang-Zhang.
(Note that Conjecture \ref{Beilinson-Bloch3} is true when $m=1$.
See Remark \ref{Remark:r=1}.)
This case is called \textit{Kudla's modularity conjecture}, stated by Kudla in  \cite[Section 3.2, Problem 1]{SpecialKudla} and proved unconditionally by Yuan-Zhang-Zhang in \cite[Theorem 1.2]{YZZ}.
However, they also assumed the absolutely convergence of the generating series for $r>1$.
\end{rem}

\begin{rem}
We do not know the absolute convergence of the generating series  $\ell(Z_{\phi_f})(\tau)$.
When $F=\Q$ and $d=e=1$, Bruinier and Westerholt-Raum proved unconditionally that $\ell(Z_{\phi_f})(\tau)$ is absolutely convergent for any $\ell$ in \cite[Corollary 1.4]{Bruinier}.
\end{rem}

\begin{rem}
\label{Kudla}
In his recent preprint \cite{Remarks}, Kudla proved the absolute convergence and the modularity of generating series in the same setting as ours, assuming Conjecture \ref{Beilinson-Bloch3} for Shimura varieties associated with quadratic spaces of sufficiently large rank.
\end{rem}
\subsection{Outline of the proof of Theorem \ref{MainTheorem1} and Theorem \ref{MainTheorem2}}
We mostly follow the strategy of Yuan-Zhang-Zhang \cite{YZZ}.
However, we have to treat higher codimensional cycles rather than 1 even in the case of $r=1$ different from \cite{YZZ}, so we need algebraic geometrical consideration, such as the Beilinson-Bloch conjecture.

First, we shall prove Theorem \ref{MainTheorem1} (2).
To prove Theorem \ref{MainTheorem1} (2), we calculate the cohomology of the Shimura variety $M_{K_f}$.
By the Matsushima formula, we conclude
\[H^{2e-1}(M_{K_f},\C)=0.\]
Since we are assuming Conjecture \ref{Beilinson-Bloch3} holds for $M_{K_f}$ and $m=e$, the cycle map tensored with $\C$
\[cl^e_{\C}\colon\CH^e(M_{K_f})_{\C}\to H^{2e}(M_{K_f},\C)\]
is injective.
Hence every $\C$-linear map $\CH^e(M_{K_f})_{\C}\to\C$ is extended to a $\C$-linear map $H^{2e}(M_{K_f},\C)\to\C$.
We can deduce Theorem \ref{MainTheorem1} (2) from the results of Kudla \cite[Section 5.3]{Remarks} and Rosu-Yott \cite[Theorem 1.1]{RosuYott}.

Then we shall prove Theorem \ref{MainTheorem2} (2) by the intersection formula \cite[Proposition 2.6]{YZZ} and the pull-back formula \cite[Proposition 3.1]{YZZ}.

Finally, we deduce Theorem \ref{MainTheorem1} (1) from Theorem \ref{MainTheorem1} (2) and deduce Theorem \ref{MainTheorem2} (1) from Theorem \ref{MainTheorem2} (2).
When $r\geq 2$, we prove Theorem \ref{MainTheorem1} (1) and Theorem \ref{MainTheorem2} (1) by induction on $r$.
We put $J\defeq\begin{pmatrix}
0 & -1_r \\
1_r & 0
\end{pmatrix}
\in\GL_{2r}(F)$.
The symplectic group
\[\Sp_{2r}(F)\defeq\bigg\{ g\in\GL_{2r}(F)
\ \bigg|\ \transp{g}Jg=J \bigg\}\]
is generated by the Siegel parabolic subgroup $P(F)$ and an element $w_1\in\Sp_{2r}(F)$.
Here
\[P(F)\defeq\bigg\{ \begin{pmatrix}
A & B \\
0 & D
\end{pmatrix}\in\Sp_{2r}(F)
\bigg\}\]
and $w_1$ is the image of
$\begin{pmatrix}
0 &-1 \\
1 & 0
\end{pmatrix}$
by the injection
\begin{align*}
\SL_{2}&\hookrightarrow\Sp_{2r} \\
\begin{pmatrix}
a & b \\
c & d
\end{pmatrix}&\mapsto
\begin{pmatrix}
a & 0 & b & 0\\
0 & 1_{r-1} & 0 & 0\\
c & 0 & d & 0\\
0 & 0 & 0 & 1_{r-1}\\
\end{pmatrix}.
\end{align*}
We consider a function $Z_{\phi_f}(g')$ on the metaplectic group $\Mp_{2r}(\A_F)$ which is a lift of $Z_{\phi_f}(\tau)$.
It suffices to prove that the function $Z_{\phi_f}(g')$ is invariant under the action of $P(F)$ and $w_1$.
A direct calculation shows the invariance under the action of an element of $P(F)$.
To prove the invariance under $w_1$, we use the Poisson summation formula to reduce to the case $r=1$.

\subsection{Outline of this paper}
In Section \ref{Section:SomeResults},
we recall some facts about special cycles and Weil representations.
In Section \ref{Section:r=1proof},
we calculate the cohomology of a Shimura variety and prove Theorem \ref{MainTheorem1} (2) and Theorem \ref{MainTheorem2} (2).
Finally, in Section \ref{Section:generalproof},
we complete a proof of Theorem \ref{MainTheorem1} and Theorem \ref{MainTheorem2}.

\section{Special cycles and Weil representations}
\label{Section:SomeResults}

In this section, we recall and extend some properties of special cycles in Chow groups.
We also note about Weil representations since in the proof of our main results, we use the function on $\Mp_{2r}(\A_F)$, the metaplectic double cover of $\Sp_{2r}(\A_F)$, lifting  $Z_{\phi_f}(\tau)$.
For more details, see \cite{YZZ}.

\subsection{Special cycles}
Let $W$ be an $F$-vector subspace of
\[\widehat{V}\defeq V\otimes_{\Q}\A_f.\]
We say $W$ is \textit{admissible} if the restriction of the inner product to $W$ is $F$-valued and totally positive.
We say an element $x=(x_1,\dots,x_r)\in\widehat{V}^r$ is \textit{admissible} if the $F$-subspace of $\widehat{V}$ spanned by $x_1,\dots,x_r$ is admissible.
The following lemma shows admissibility is useful for description of special cycles.

\begin{lem}
\label{Lemma:adm}
An $F$-subspace $W$ of $\widehat{V}$ is admissible if and only if
there exists an $F$-subspace $W'$ of $V$ and $g\in G(\A_f)$ such that
$W=gW'$.
\end{lem}
\begin{proof}
See \cite[Lemma 2.1]{YZZ}.
\end{proof}

By the above lemma, for an admissible subspace $W=g^{-1}W'$, we define
$Z(W)_{K_f}\defeq Z(W',g)_{K_f}$.
In the same way, for an admissible element $x=g^{-1}x'$, we write
$Z(x)_{K_f}\defeq Z(x',g)_{K_f}$.
Moreover, for $\tau\in(\H_r)^d$ and $g'\in\Mp_{2r}(\A_F)$, we get the following descriptions:
\begin{align*}
Z_{\phi_f}(\tau)&=\sum_{\substack{x\in K_f\backslash \widehat{V}^r \\ \mathrm{admissible}}}\phi_f(x)Z(x)_{K_f}q^{T(x)}\\
Z_{\phi_f}(g')&=\sum_{\substack{x\in K_f\backslash \widehat{V}^r \\ \mathrm{admissible}}}\omega_f(g_f')\phi_f(x)Z(x)_{K_f}W_{T(x)}(g_{\infty}')
\end{align*}
By \cite[Proposition 2.2]{YZZ}, the scheme-theoretic intersection of
two cycles $Z(W_1)_{K_f}$ and $Z(W_2)_{K_f}$ is the union of $Z(W)$ indexed by
admissible classes $W$ in
\[K_f\backslash (K_fW_1+K_fW_2).\]
We shall investigate the intersection of two cycles in the Chow group.



\begin{prop}
\label{Prop:dimension}
The intersection of two cycles $Z(W_1)_{K_f}$ and $Z(W_2)_{K_f}$ in the Chow group are proper if and only $k_1W_1\cap k_2W_2=0$ for all admissible classes $k_1W_1+k_2W_2$.
\end{prop}
\begin{proof}
We recall that $\dim Z(W_i)_{K_f}=e(n-\dim W_i)$.
The intersection is proper if and only if the following inequality holds:

\begin{eqnarray*}
\dim (Z(W_1)_{K_f}\cap Z(W_2)_{K_f})
&\leq&\dim Z(W_1)_{K_f} + \dim Z(W_2)_{K_f} - \dim M_{K_f}\\
&=&e(n-(\dim W_1+\dim W_2))
\end{eqnarray*}
On the other hand,
\[Z(W_1)_{K_f}\cap Z(W_2)_{K_f}=\sum_{\substack{W\in K_f\backslash (K_fW_1+K_fW_2) \\ \mathrm{admissible}}}Z(W)_{K_f}\]
and
\[\dim Z(k_1W_1+k_2W_2)_{K_f}=e(n-(\dim k_1W_1+\dim k_2W_2)+\dim (k_1W_1\cap k_2W_2)).\]

Therefore the above inequality holds if and only if $k_1W_1\cap k_2W_2=0$ for all admissible classes $k_1W_1+k_2W_2$.

\end{proof}
\begin{prop}
\label{Prop:intersection}
The intersection of two cycles $Z(W_1)_{K_f}$ and $Z(W_2)_{K_f}$ in the Chow group is given by
the sum of $Z(W)_{K_f}$ indexed by
admissible classes $W$ in
\[K_f\backslash (K_fW_1+K_fW_2).\]
\end{prop}
\begin{proof}
In the same way as the proof of \cite[Proposition 2.6]{YZZ}, we have to check that if $\dim W_2=1, Z(W_1)_{K_f}\subset Z(W_2)_{K_f}$, then $Z(W_1)_{K_f} \cdot Z(W_2)_{K_f}=Z(W_1)_{K_f} \cdot \mathrm{c_1}(\L)$.
Let $\N$ be the restriction of the normal bundle $\N_{Z(W_2)_{K_f}}({M_{K_f}})$ to $Z(W_2)_{K_f}$.
Now,
\[Z(W_1)_{K_f} \cdot Z(W_2)_{K_f}=\mathrm{c_e}(\N)\cap Z(W_2)_{K_f}\]
and by the calculation of normal bundles in \cite[Chapter 4]{Remarks}, we have
\[\mathrm{c_e}(\N)\cap Z(W_2)_{K_f}=Z(W_1)_{K_f} \cdot \mathrm{c_1}(\L).\]
\end{proof}

\subsection{The pull-back formula}
\label{Subsection:Formula}
Here we study the behavior of special cycles under the pull-back map.
Let $W\subset V$ be a totally positive $F$-vector subspace.
There exists a natural morphism
\[i_W\colon M_{K_f,W}\to M_{K_f},\]
which is a closed embedding if $K_f$ is sufficiently small.
Therefore we get a pull-back map of Chow groups:
\[i_W^*\colon\CH^{er}(M_{K_f})\to\CH^{er}(M_{K_f,W}).\]
For $g'=(g_f',g_{\infty}')\in\Mp_{2r}(\A_F)=\Mp_{2r}(\A_{F,f})\times\Mp_{2r}(F_{\infty})$, we define
\[i_W^*(Z_{\phi_f})(g')\defeq\sum_{\substack{x\in K_f\backslash \widehat{V}^r \\  \mathrm{admissible}}}\omega_f(g_f')\phi_f(x)i_W^*(Z(x)_{K_f})W_{T(x)}(g_{\infty}').\]
For a Bruhat-Schwartz function $\phi_{2,f}\in\S((\widehat{W^{\perp}})^r)$, the theta function is defined by
\[\theta_{\phi_{2,f}}(g')\defeq\sum_{z\in W^r}\omega_{\A}(g')(\phi_{2,f}\otimes\varphi_+^d)(z).\]
\begin{prop}
\label{Prop:pull-back}
For a $K_f$-invariant Bruhat-Schwartz function
\[\phi_f=\phi_{1,f}\otimes\phi_{2,f}\in \S(\widehat{V}^r)=\S(\widehat{W}^r)\otimes_{\C}\S((\widehat{W^{\perp}})^r),\]
we have
\[i_W^*(Z_{\phi_f}(g'))=Z_{\phi_{1,f}}(g')\theta_{\phi_{2,f}}(g').\]

\end{prop}

\begin{proof}
Proposition \ref{Prop:intersection} implies that the assertion is proved by the same way as \cite[Proposition 3.1]{YZZ}.
\end{proof}

\subsection{Weil representations}
Let $\psi\colon F\backslash\A_F\to \C^{\times}$ be the composite of the trace map  $F\backslash\A_F\to\Q\backslash\A$ and the usual additive character
\begin{align*}
\Q\backslash\A&\to\C^{\times}\\
(x_v)_v&\mapsto\exp(2\pi\sqrt{-1}(x_{\infty}-\sum_{v<\infty}\overline{x_v})),
\end{align*}
where $\overline{x_v}$ is the class of $x_v$ in $\Q_p/\Z_p$.

Let $W$ be a symplectic vector space of dimension $2r$ over $F$.
We consider a reductive dual pair $(\O(V),\Sp(W))$ in $\Sp(V\otimes_{F} W)$.
Then we get a Weil representation $\omega$ which is the action of $\Mp_{2r}(\A_F)\times\O(V(\A_F))$ on $\S(V(\A_F)^r)$.
Let $\omega_f$ and $\omega_{\A}$ the action of $\Mp_{2r}(\A_{F,f})$ on $\S(V(\A_{F,f})^r)$ and $\Mp_{2r}(\A_F)$ on $\S(V(\A_F)^r)$ respectively.
Here we put $F_{\infty}\defeq F\otimes_{\Q}\R\cong\prod_{i=1}^d\R$.

Now, we introduce the degenetate Whittaker function.
We shall use the same notation as in \cite{OrthogonalKudla}.
Let $(V_+,(\ ,\ )_+)$ be a positive definite quadratic space of dimension $n+2$ over $\R$ and $\omega_+$ be an action of $\Mp_{2r}(\R)$ to $\S(V_+^r)$.
Let $\varphi_+\in \S(V_+^r)$ be the Gaussian defined by
\[\varphi_+(x)\defeq\exp(-\pi((x_1,x_1)_++\dots+(x_r,x_r)_+))\quad(x=(x_1,\dots,x_r)\in V_+^r).\]
Let \[K_{\infty}\defeq\bigg\{ \begin{pmatrix}
p & q \\
-q & p
\end{pmatrix}
\in \Sp_{2r}(\R)
\ \bigg|\ (p+\sqrt{-1}q)\,\transp{(p-\sqrt{-1}q)}=1_r\bigg\}.\]
be the maximal compact subgroup in $\Sp_{2r}(\R)$ and $\widetilde{K}_{\infty}$ be the inverse image of $K_{\infty}$ in $\Mp_{2r}(\R)$.
Then the function $\varphi_+$ is an eigenvector with respect to the Weil representation $\omega_+$:
\[\omega_+(k)\varphi_+=\det(k)^{(n+2)/2}\varphi_+\quad (k\in \widetilde{K}_{\infty}).\]
For a symmetric matrix $T\in\Sym_r(\R)$ of size $r\times r$, we take an element $x\in V_+^r$ satisfying $\frac{1}{2}(x,x)_+=T$.
For $g_{\infty}\in\Mp_{2r}(\R)$, we define the degenerate Whittaker function by
\[W_{T}(g_{\infty})\defeq\omega_+(g_{\infty})\varphi_+(x).\]
For $T\in\Sym_{r}(F)$ and
\[g_{\infty}'=(g_{\infty,1}',\dots,g_{\infty,d}')\in\Mp_{2r}(F_{\infty})\cong\prod_{i=1}^{d}\Mp_{2r}(\R),\]
we set
\[W_{T}(g_{\infty}')\defeq W_{T^{\sigma_1}}(g_{\infty,1}')\cdots W_{T^{\sigma_d}}(g_{\infty,d}').\]

For $g'=(g_f',g_{\infty}')\in\Mp_{2r}(\A_F)=\Mp_{2r}(\A_{F,f})\times\Mp_{2r}(F_{\infty})$, we put
\[{Z_{\phi_f}(g')}
\defeq\sum_{x\in G(\Q)\backslash V^r}\sum_{g\in G_x(\A_f)\backslash G(\A_f)/K_f}\omega_f(g_f')\phi_f(g^{-1}x)Z(x,g)_{K_f}W_{T(x)}(g_{\infty}').\]
By the Fourier expansion, we consider ${Z_{\phi_f}(g')}$ as a formal power series with coefficients in $\CH^{er}(M_{K_f})_{\C}$.
Therefore, the modularity of the generating series $Z_{\phi_f}(\tau)$ is equivalent to the left $\Sp_{2r}(F)$-invariance of the function $Z_{\phi_f}(g')$ on $\Mp_{2r}(\A_F)$.

\section{Proof of Theorem \ref{MainTheorem1} (2) and Theorem \ref{MainTheorem2} (2)}
\label{Section:r=1proof}

\subsection{Cohomology of Shimura varieties of orthogonal type}
\label{Subsection:cohomology}
In this subsection, we shall prove if $n\geq 3$, then
\[H^{2e-1}(M_{K_f},\C)=0.\]

Recall that we have
\[M_{K_f}(\C)\cong\coprod_{\Gamma}X_{\Gamma},\]
where $X_{\Gamma}=\Gamma\backslash D$ and $\Gamma$ is a cocompact congruence subgroup of
\[\SO_0({V\otimes_{\Q}\R})\cong\SO_0(n,2)^e\times\SO(n+2)^{d-e}.\]
Here $\SO_0(V\otimes_{\Q}\R),\SO_0(n,2)$ denote the identity components of $\SO(V\otimes_{\Q}\R),\SO(n,2)$, respectively.
Therefore, it is enough to show $H^{2e-1}(X_{\Gamma},\C)=0$.

We put $G'\defeq(\Res_{F/\Q}\SO(V))(\R)$ and $\g'\defeq(\Lie G')\otimes_{\R}\C$.
We put
\[G'_i\defeq\begin{cases}\SO_0(n,2) & (1\leq i \leq e)\\\SO(n+2) & (e+1\leq i \leq d),\end{cases}\]
$\g'_i\defeq\Lie(G_i')\otimes_{\R}\C$.
We also put
\[K'_i\defeq\begin{cases}\SO(n)\times\SO(2) & (1\leq i \leq e)\\\SO(n+2) & (e+1\leq i \leq d)\end{cases}\]
and $K'\defeq K_1'\times\cdots\times K_d'$.
By the Matsushima formula, we can write the cohomology of $X_{\Gamma}$ as follows:
\[H^{2e-1}(X_{\Gamma},\C)\cong\bigoplus_{\pi\in \widehat{G'(\R)}}\mathrm{Int}_{\Gamma}(\pi)\otimes_{\C} H^{2e-1}(\g',K';\pi).\]
Here $\widehat{G'(\R)}$ is the set of irreducible unitary representations of $G'(\R)$, $\mathrm{Int}_{\Gamma}(\pi)$ is the one appearing in the decomposition
\[L^2(\Gamma\backslash G'(\R))\cong\mathop{\widehat{\bigotimes}}_{\pi \in \widehat{G'(\R)}}\mathrm{Int}_{\Gamma}(\pi)\otimes\pi.\]
Since $\pi$ is an irreducible unitary representation, $\pi$ decomposes as $\pi\cong\widehat{\otimes}_{i=1}^d\pi_i$.
See \cite[Theorem 1.2]{GourevitchKemarsky}.

Then by the K\"{u}nneth formula \cite[Section 1.3]{BorelWallach}, we have
\begin{equation}
\label{coh1}
H^{2e-1}(\g',K';\pi)\cong\bigoplus_{i_1+\dots+i_d=2e-1}\bigotimes_{k=1}^{d} H^{i_k}(\g'_k,K'_k;\pi_k).
\end{equation}
For $e+1\leq i\leq d$, we have $H^j(\g'_i,K'_i;\pi_i)=0$ for any $j\geq 1$ since $\SO(n+2)$ is compact.
Therefore (\ref{coh1}) can be written as follows:
\begin{eqnarray}
\label{coh2}
& & H^{2e-1}(\g',K';\pi) \\
& \cong & \biggr(\bigoplus_{i_1+\dots+i_e=2e-1}\bigotimes_{k=1}^{e}H^{i_k}(\g'_k,K'_k;\pi_k)\biggr)\otimes_{\C} \bigotimes_{k=e+1}^{d} H^0(\g'_k,K'_{k};\pi_{k}).
\nonumber
\end{eqnarray}

\begin{lem}
\label{Lem:VZK}
Assume $n\geq 3$.
For $1\leq i\leq e$, if $\pi_i$ is non-trivial, then we have
\[H^j(\g'_i,K'_i;\pi_i)=0\]
for $j=0,1$.
\end{lem}
\begin{proof}
See \cite[Theorem 8.1]{VZ} and the Kumaresan vanishing theorem \cite[Section 3]{Kumaresan}.
\end{proof}

In the rest of this subsection, we assume $n\geq 3$.
Then, by Lemma \ref{Lem:VZK}, we can write (\ref{coh2}) as follows:
\begin{eqnarray}
\label{coh3}
& & H^{2e-1}(\g',K';\pi) \\
& \cong & \biggr(\bigoplus_{\substack{i_1+\dots+ i_e=2e-1 \\ 1\leq \exists j\leq e,\ \pi_j\mathrm{:trivial}}}\bigotimes_{k=1}^{e}H^{i_k}(\g'_k,K'_k;\pi_k)\biggr)\otimes_{\C} \bigotimes_{k=e+1}^{d} H^0(\g'_k,K'_{k};\pi_{k}).
\nonumber
\end{eqnarray}

Here we need the following lemma.
\begin{lem}
\label{Lem:trivial}
Let $L$ be a totally real number field, $V$ a non-degenerate quadratic space of dimension $n+2$ over $L$, and $\pi\cong\otimes_v\pi_v$ an automorphic representation of $\SO(V)(\A_L)$.
If there exists an archimedean place $w$ such that $\SO(V)(L_w)\cong\SO(n,2)$ and the restriction of $\pi_w$ to the identity component of $\SO(V)(L_w)$ is the trivial representation, then $\pi_v$ is a character for any places $v$.
\end{lem}
\begin{proof}
See \cite[Lemma 3.24]{Manin}.
\end{proof}

The connected Lie group $\SO_0(n,2)$ is semisimple and has no compact factor.
Hence $\pi_i$ is the trivial representation for every $1\leq i\leq e$.
See \cite[Section 4.3.2, Example 4]{Warner}.

Then (\ref{coh3}) becomes as follows:
\begin{eqnarray}
\label{coh4}
& & H^{2e-1}(\g',K';\pi) \\
& \cong & \biggr(\bigoplus_{i_1+\dots+ i_e=2e-1}\bigotimes_{k=1}^{e}H^{i_k}(\g'_k,K'_k;1)\biggr)\otimes_{\C} \bigotimes_{k=e+1}^{d} H^0(\g'_k,K'_{k};\pi_{k}).
\nonumber
\end{eqnarray}

Finally by \cite[Section 5.10]{Hodgetype}, for $1\leq i\leq e$, we have $H^s(\g'_i,K'_i;1)=0$  if $s$ is odd.

Thus, \[H^{2e-1}(\g',K';\pi)=0.\]
Combining the above results, we get the following theorem.
\begin{thm}
\label{Theorem:CohomologyIsVanish}
Assume $n\geq 3$.
Then we have \[H^{2e-1}(M_{K_f},\C)=0.\]
\end{thm}

\begin{cor}
\label{Corollary:Injectivity}
Assume $n\geq 3$.
Assume moreover that Conjecture \ref{Beilinson-Bloch3} holds for $M_{K_f}$ and $m=e$.
Then the cycle map tensored with $\C$
\[cl^e_{\C}\colon\CH^e(M_{K_f})_{\C}\to H^{2e}(M_{K_f},\C)\] is injective.
\end{cor}

\subsection{Proof of Theorem \ref{MainTheorem1} (2) and Theorem \ref{MainTheorem2} (2)}
\label{Subsection:r=1proof}
If $n\geq 3$, by Corollary \ref{Corollary:Injectivity}, the assertion follows from the results of Kudla \cite[Section 5.3]{Remarks} and Rosu-Yott \cite[Theorem 1.1]{RosuYott}.

If $n\leq 2$, we take a totally positive quadratic space $W$ of dimension $\geq 3$ over $F$.
We embed $V$ into $V\oplus W$.
For any $K_f$-invariant Bruhat-Schwartz functions $\phi_f\in\S(\widehat{V})$ and $\phi_f'\in\S(\widehat{W})$, using the pull-back formula (Proposition \ref{Prop:pull-back}), we get
\[i^*(Z_{\phi_f\otimes\phi_f'})(g')=Z_{\phi_f}(g')\theta_{\phi_f'}(g')\]
for any $g'\in\Mp_{2}(\A_F)$.
Since $Z_{\phi_f\otimes\phi_f'}(g')$ and $\theta_{\phi_f'}(g')$ are absolutely convergent and left $\SL_2(F)$-invariant, we conclude that $Z_{\phi_f}(g')$ is absolutely convergent and left $\SL_2(F)$-invariant.

The proof of Theorem \ref{MainTheorem1} (2) and Theorem \ref{MainTheorem2} (2) is complete.

\section{Proof of Theorem \ref{MainTheorem1} (1) and Theorem \ref{MainTheorem2} (1)}
\label{Section:generalproof}

\subsection{Invariance under the Siegel parabolic subgroup}
For $a\in\GL_{r}(F)$ and $u\in\Sym_r(F)$, we put $m(a)\defeq
\begin{pmatrix}
a & 0 \\
0 & \transp{a}^{-1}
\end{pmatrix}$ and $n(u)\defeq
\begin{pmatrix}
1 & u \\
0 & 1
\end{pmatrix}$.
The elements $m(a)$ and $n(u)$ generate the Siegel parabolic subgroup $P(F)\subset \Sp_{2r}(F)$.

For $g'\in\Mp_{2r}(\A_F)$, its infinity component in $\Mp_{2r}(F_{\infty})\cong\prod_{i=1}^d\Mp_{2r}(\R)$ is denoted by $g_{\infty}'=(g_{\infty,1}',\dots,g_{\infty,d}')$.
For $1\leq i\leq d$, we consider the Iwasawa decomposition of $g_{\infty,i}'$:
\[g_{\infty,i}'=
\begin{pmatrix}
1 & s_i \\
0 & 1
\end{pmatrix}
\begin{pmatrix}
t_i & 0 \\
0 & \transp{t}_i^{-1}
\end{pmatrix}
k_i\quad (s_i\in\Sym_r(\R),t_i\in\GL^+_r(\R),k_i\in \widetilde{K}_{\infty}).\]

For $T\in\Sym_{r}(\R)$, the degenerate Whittaker function satisfies the following formula:
\[W_{T}(g_{\infty,i}')=|\det(s_i)|^{(n+2)/4}\exp(2\pi \sqrt{-1}(\Tr(\tau_i T)))\det(k_i)^{(n+2)/2}\]
where $\tau_i=s_i+it_i\transp{t}_i$.

By \cite[Part I, Section 1]{Central}, $n(u)$ acts as follows:
\[\omega_f(n(u)_f)\phi_f(x)=\psi_f(\Tr(u_fT(x)))\phi_f(x).\]
Thus, we have
\begin{eqnarray*}
& &  \omega_f(n(u)_fg_f')(\phi_f)(x)Z(x)_{K_f}\prod_{i=1}^{d}W_{T(x)}(n(u)_{\infty,i}g_{\infty,i}') \\ &=&\psi_f(\Tr(u_fT(x_f)))\omega_f(g_f')\phi_f(x)Z(x)_{K_f}\psi_{\infty}(\sum_{i=1}^{d}\Tr(u_{i,\infty}T(x)_{\infty,i}))\prod_{i=1}^{d}W_{T(x)}(g_{\infty,i}') \\
&=&\psi(\Tr(uT(x)))\omega_f(g_f')\phi_f(x)Z(x)_{K_f}\prod_{i=1}^{d}W_{T(x)}(g_{\infty,i}') \\
&=&\omega_f(g_f')(\phi_f)(x)Z(x)_{K_f}\prod_{i=1}^{d}W_{T(x)}(g_{\infty,i}').
\end{eqnarray*}

Therefore, we have the term-wise invariance under $n(u)$:
\[\omega_f(n(u)_fg_f')(\phi_f)(x)Z(x)_{K_f}W_{T(x)}(n(u)_{\infty}g_{\infty}')=\omega_f(g_f')(\phi_f)(x)Z(x)_{K_f}W_{T(x)}(g_{\infty}').\]

By the same way, we have
\[\omega_f(m(a)_fg_f')(\phi_f)(x)Z(x)_{K_f}W_{T(x)}(m(a)_{\infty}g_{\infty}')=\omega_f(g_f')(\phi_f)(xa)Z(x)_{K_f}W_{T(xa)}(g_{\infty}')\]
for any $a\in\GL_r(F)$.

On the other hand. we have $U(x)=U(xa)$, so $Z_{\phi_f}(x)=Z_{\phi_f}(xa)$.
Therefore, combining the above calculation and the fact $Z_{\phi_f}(x)=Z_{\phi_f}(xa)$, we conclude that
\begin{eqnarray*}
Z_{\phi_f}(\omega_f(m(a))g')
&=& \sum_{\substack{x\in K_f\backslash \widehat{V}^r \\ \mathrm{admissible}}}\omega_f(g_f')\phi_f(xa)Z(xa)_{K_f}W_{T(xa)}(g_{\infty}') \\
&=& \sum_{\substack{x\in K_f\backslash \widehat{V}^r \\ \mathrm{admissible}}}\omega_f(g_f')\phi_f(x)Z(x)_{K_f}W_{T(x)}(g_{\infty}') \\
&=& Z_{\phi_f}(g').
\end{eqnarray*}
This shows $Z_{\phi_f}(g')$ is invariant under the action of the Siegel parabolic subgroup $P(F)$.

\subsection{Invariance under $w_1$}
By Proposition \ref{Prop:pull-back}, we get following expression. (For details, see  \cite{YZZ}.)
\[Z_{\phi_f}(\tau)=\sum_{\substack{y\in K\backslash \widehat{V}^{r-1} \\ \mathrm{admissible}}}\sum_{x_2\in Fy}\sum_{\substack{x_1\in K_{f,y}\backslash y^{\perp} \\ \mathrm{admissible}}}\phi_f(x_1+x_2,y)Z(x_1)_{K_{f,y}}q^{T(x_1+x_2,y)}\]
Thus we have
\begin{eqnarray*}
Z_{\phi_f}(g')
&=&\sum_{\substack{y\in K_f\backslash \widehat{V}^{r-1} \\ \mathrm{admissible}}}\sum_{x_2\in Fy}\sum_{\substack{x_1\in K_{f,y}\backslash y^{\perp} \\ \mathrm{admissible}}}\omega_f(g_f')\phi_f(x_1+x_2,y)Z(x_1)_{K_{f,y}}W_T(x_1+x_2,y)(g_{\infty}')\\
&=&\sum_{\substack{y\in K_f\backslash \widehat{V}^{r-1} \\ \mathrm{admissible}}}\sum_{x_2\in Fy}\sum_{\substack{x_1\in K_{f,y}\backslash y^{\perp} \\ \mathrm{admissible}}}\omega_{\A}(g')(\phi_f\otimes\varphi_+^d)(x_1+x_2,y)Z(x_1)_{K_{f,y}}.
\end{eqnarray*}
Here $\phi^x(x,y)$ is the partial Fourier transformation with respect to the first coordinate.

Now, by Theorem \ref{MainTheorem1} (2) and Theorem \ref{MainTheorem2} (2),
we have
\begin{eqnarray*}
Z_{\phi_f}(w_1g')
&=&\sum_{\substack{y\in K_f\backslash \widehat{V}^{r-1} \\ \mathrm{admissible}}}\sum_{x_2\in Fy}\sum_{\substack{x_1\in K_{f,y}\backslash y^{\perp} \\ \mathrm{admissible}}}\omega_{\A}(g')(\phi_f\otimes\varphi_+^d)^{x_2}(x_1+x_2,y)Z(x_1)_{K_{f,y}}.
\end{eqnarray*}
Here we use the fact that \[\omega_{\A}(w_1)(\phi_f\otimes\varphi_+^d)(x,y)=(\phi_f\otimes\varphi_+^d)^{x}(x,y).\]

By the Poisson summation formula, this equals to
\[\sum_{\substack{y\in K_f\backslash \widehat{V}^{r-1} \\ \mathrm{admissible}}}\sum_{x_2\in Fy}\sum_{\substack{x_1\in K_{f,y}\backslash y^{\perp} \\ \mathrm{admissible}}}\omega_{\A}(g')(\phi_f\otimes\varphi_+^d)(x_1+x_2,y)Z(x_1)_{K_{f,y}},\]
which coincides with the definition of $Z_{\phi_f}(g')$.
Therefore, we get
\[Z_{\phi_f}(w_1g')=Z_{\phi_f}(g').\]
This shows the function $Z_{\phi_f}(g')$ is invariant under $w_1$.

The proof of Theorem \ref{MainTheorem1} and Theorem \ref{MainTheorem2} is complete.

\subsection*{Acknowledgements}
The author would like to express his gratitude to his adviser, Tetsushi Ito, for helpful and insightful comments.
The author would also like to offer special thanks to Stephan Kudla for invaluable comments and warm encouragement, Shuji Horinaga for giving advice on the Weil representations, and Kazuhiro Ito for suggestion on the calculation of the Lie algebra cohomology.

\end{document}